\documentclass[conference] {IEEEtran}
\usepackage[utf8]{inputenc}
\usepackage[reqno]{amsmath}
\usepackage{amssymb}
\usepackage{amsthm}
\usepackage{amscd}
\usepackage{enumerate}

\usepackage{graphics,xy,epsfig,picture,epic}

\def\ee{{\mathcal E}}

\def\rr{{\mathcal R}}

\def\vv{{\mathcal V}}

\def\ww{{\mathcal W}}

\def\eps{\varepsilon}
\def\dst{\displaystyle}

\def\C{{\mathbb{C}}}

\def\R{{\mathbb{R}}}

\def\T{{\mathbb{T}}}

\newcommand{\norm}[1]{{\left\|{#1}\right\|}}

\newcommand{\abs}[1]{{\left|{#1}\right|}}
\newcommand{\scal}[1]{{\left\langle{#1}\right\rangle}}

\theoremstyle{plain}

\newtheorem{lemma}{Lemma}[section]
\newtheorem{proposition}[lemma]{Proposition}
\newtheorem{theorem}[lemma]{Theorem}

\newtheorem{conjecture}[lemma]{Conjecture}
\newtheorem{question}[lemma]{Question}

\theoremstyle{definition}
\newtheorem{definition}[lemma]{Definition}

\theoremstyle{remark}

\numberwithin{equation}{section}

\begin{document}

\title{Phase retrieval for solutions of\\ the Schr\"odinger equations}

\author{
\IEEEauthorblockN{Philippe {\sc Jaming}}

\IEEEauthorblockA{Univ. Bordeaux, IMB, UMR 5251\\ F-33400 Talence, France\\\
CNRS, IMB, UMR 5251\\ F-33400 Talence, France\\
{\tt Philippe.Jaming@math.u-bordeaux.fr}}}

\maketitle

\begin{abstract}
In this note we present several questions about the phase retrieval problem for the
Schr\"odinger equation. Some partial answers are given as
well as some of the heuristics behind these questions.
\end{abstract}


\section{Introduction}

The phase retrieval problem consists in reconstructing a function from its modulus or the modulus of some
transform of it (frame coefficient, Fourier transform,...) and some structural information on the function ({\it e.g.}
to be compactly supported). Such a problem occurs in many scientific fields:
microscopy, holography, crystallography, neutron radiography,
optical coherence tomography, optical design, radar signal processing
and quantum mechanics to name a few. 

As for quantum mechanics, the most popular phase retrieval problem in the area is
Pauli's problem which asks whether a function $f\in L^2(\R^d)$
is uniquely determined (up to a global phase factor) by its modulus $|f|$ and the modulus of its Fourier transform $|\widehat{f}|$. In other words we are asking whether
$f,g\in L^2(\R^d)$ such that $|f|=|g|$ and $|\widehat{f}|=|\widehat{g}|$ implies
$f=cg$ with $c\in \T:=\{z\in\C\,:|z|=1\}$ (a global phase factor). The answer to this question is well known to be negative and there are now many counter-examples ({\it see e.g.} \cite{AJ} and references therein).

It was further conjectured by Wright that there exists a third unitary operator $U\,:L^2(\R^d)\to L^2(\R^d)$ such that $|f|,|\widehat{f}|,|Uf|$ uniquely determine $f$.
As far as we know, this question is open at the time of writing this note,
but if $U$ is only asked to be self adjoint, then there are examples for which the answer
is positive ({\it see} \cite{JR} for more on the problem and the explanation on
why the assumption of $U$ being unitary is crucial).

Our aim here is in some sense more modest as we are looking for a one parameter
family of unitary operators (actually a semi-group) such that $|U_\alpha f|$
uniquely determines $f$ up to a global phase factor. As far as we know,
the only family known so far that answers this question positively is
given by the fractional Fourier transform \cite{Ja}.
It turns out that, thanks to minor renormalization computations, the result can be
reformulated in terms of solutions of Schr\"odinger equations:

\begin{proposition}[Jaming \cite{Ja}]\label{prop1}
Let $u_0,v_0\in L^2(\R^d)$ and let $u,v$ be the solution of the
{\em free Schr\"odinger equation}
$$
\left\{\begin{matrix}i\partial_tu(t,x)=-\Delta u(t,x)\\ u(0,x)=u_0\end{matrix}\right.
\ ,\quad \left\{\begin{matrix}i\partial_tv(t,x)=-\Delta v(t,x)\\ v(0,x)=v_0\end{matrix}\right..
$$
If $|u(t,x)|=|v(t,x)|$ for every $t\in\R$ and every $x\in\R^d$
then there exists $c\in\R$ such that $u_0=cv_0$.
\end{proposition}

This proposition is very simple and most of the content of \cite{Ja} deals with the following questions linked to Wright's conjecture: can one restrict the set of times $t$
needed to obtain the same answer, at least if one restrict $u_0$ and/or $v_0$ to be in
some set of functions with some additional structure. For instance, it is shown that
if $u_0,v_0$ belong to the spectral subsets of the Laplacian (that is to the Paley-Wiener spaces)
then an explicit discrete set of time suffices.

\smallskip

Our aim here is to start a different line of investigation. The previous proposition
is restricted to the free Schr\"odinger equation on $\R^d$. It is natural
to try to extend this result to a more general setting in which the Schr\"odinger equation
has a natural physical meaning:

Let $\Omega$ be any structure on which the Laplace operator has a natural analogue. For instance

-- The natural Euclidean metric on $\R^d$ could be deformed which would lead to $-\Delta$
to be replaced by a more general elliptic operator $-\mathrm{div}(A\nabla)$.

-- Instead of $\Omega=\R^d$, one could consider the Schr\"odinger equation
on an open set (say with smooth boundary) and add some boundary conditions.

-- More generally $\Omega$ could be a Riemannian manifold and $-\Delta$ the Laplace Beltrami operator.

-- In the opposite direction, $\Omega$ could be a graph (finite or not) and $-\Delta$
be the combinatorial Laplacian.

-- The intermediate situation of quantum graphs (in short, graphs in which the edges are intervals with boundary conditions on the vertices) is also of interest.

Once this is set, a further question comes immediately, even in the simplest case of $\R^d$:
what is the influence of the potential $W\,:\Omega\to\R^+$ (or $\Omega\to\R$).
In short, we are asking the following:

\begin{question}\label{Q}
Let $\Omega$ be any of the above structures and $-\Delta$ be the (positive) Laplacian associated to it. Let $W\,:\Omega\to\R^+$ be a potential. Let $u_0,v_0$ be two functions on $\Omega$
and let $u,v$ be the solutions of the same
{\em Schr\"odinger equation}
$$
\left\{\begin{matrix}i\partial_tu=-\Delta u+Wu\\ u(0,x)=u_0\end{matrix}\right.
\quad \left\{\begin{matrix}i\partial_tv=-\Delta v+Wv\\ v(0,x)=v_0\end{matrix}\right.
$$
but with initial data $u_0$ and $v_0$ respectively.
Assume that $|u(t,x)|=|v(t,x)|$ for every $t\in\R$ and every $x\in\Omega$.
Does there exists $c\in\R$ such that $u_0=cv_0$.

One may eventually ask the same question when $u_0,v_0$ belong only to a spectral
subset of $-\Delta+W$ (or any other natural restriction). In particular,
can the set of times be reduced?
\end{question}

Of course, one needs to impose some restrictions on $W$ for the question to even make sense, but we will not elaborate in this direction in this note. Our main message in this note
is that the question should be investigated in further detail and that
\begin{center}
{\em there are many more questions than answers so far!}
\end{center}

Of course, we have only asked the question of uniqueness, the other natural questions on phase 
retrieval should also be asked: stability with respect to noise, possibility of sampling, reconstruction algorithms,...

\smallskip

The remaining of the paper is organised as follows: in the next section, we give some
information on the Schr\"odinger equation on $\R^d$ that will
lead to a conjecture that we think reasonable.
The third section, we make some observation about our problem on finite graphs
that will lead to some partial results on our question in this setting.

\section{The Schr\"odinger equation on $\R^d$}
\label{sec:shr}

In this section, we make some observations on the Schr\"odinger equation on $\R^d$
and the related phase retrieval problem. We here stay on an {\em heuristic level}
and do not aim for full mathematical rigour as this would require a much longer
exposition. One may refer to \cite{LP} or to the survey \cite{Fral} for further detail.

We consider the {\it scaled} Schr\"odinger equation on $\R\times\R^d$
$$
i\eps\partial_t u^\eps=-\frac{\eps^2}{2}\Delta u^\eps+W(x)u^\eps
$$
with initial data $u^\eps(t=0,x)=u_0^\eps(x)$. Here $\eps$ stands for the
scaled Plank constant, $u^\eps=u^\eps(t,x)$ is the wave function and $W$
is a potential. From the point of view of physics, the interesting
quantity is not $u^\eps$ but the {\em position density} $|u^\eps|^2$.

Next, we introduce the {\em Wigner transform} of a function $f\in L^2(\R^d)$ by
$$
\ww^\eps[f](x,\xi)=\dfrac{1}{(2\pi)^{d/2}}\int_{\R^d}
f\left(x-\eps\frac{\eta}{2}\right)\overline{f\left(x+\eps\frac{\eta}{2}\right)}e^{i\scal{\eta,\xi}}\,\mbox{d}\eta.
$$
Recall that
$$
\ww^\eps[f](x,\xi)=\ww^\eps[g](x,\xi)\quad\Leftrightarrow\quad f=cg,\ c\in\T.
$$

Now, in the free case, the scaling is not needed and we take $\eps=1$.
If $u$ solves the Schr\"odinger equation,
$i\partial_t u=-\frac{1}{2}\Delta u$
then $\omega(t,x,\xi):=\ww^1[u(t,\cdot)](x,\xi)$ solves the
{\em transport equation}
$$
\partial_t\omega+\scal{\xi,\nabla_x\omega}=0.
$$
It follows that $\omega(t,x,\xi)=\omega(0,x-\xi t,\xi)$.
On the other hand, the well-known marginal properties
of the Wigner transform read
\begin{eqnarray*}
|u(t,x)|^2&=&\int_{\R^d}\omega(t,x,\xi)\,\mbox{d}\xi\\
&=&\int_{\R^d}\omega(0,x-\xi t,\xi)\,\mbox{d}\xi.
\end{eqnarray*}
This is nothing but the $X$-ray transform of $\omega(0,\cdot)$ through the line
$(x,0)+\R(-t,1)$. Inversion properties of the $X$-ray transform
then show that $|u(t,x)|$ determines the Wigner transform of $u_0$
and thus $u_0$ up to a global phase factor. This gives
a new (but not totally rigorous at this stage) proof of Proposition \ref{prop1}.
The original proof was slightly different. It consisted in showing that
$|u(t,x)|$ determines the restriction of the ambiguity function of $u$
to certain lines. As the
ambiguity function is the $\R^{2d}$ Fourier transform of the Wigner function,
the link between the 2 proofs is the well-known Fourier Slice Theorem for the X-ray transform.

\smallskip

In the presence of a potential $W\not=0$, the situation is more complicated and
we have no answer to our problem so far. Nevertheless we have an heuristic that
leads us to believe that the answer should be positive in many cases.
To describe it, we will now need the scaled Planck constant $\eps$. We then
consider  $\omega^\eps(t,x,\xi):=\ww^\eps[u^\eps(t,\cdot)](x,\xi)$
which still satisfies a PDE known as the {\em Wigner equation}. The formal limit when $\eps\to0$
of this equation leads to the following {\em Vlasov} equation
$$
\partial_t\omega^0+\scal{\xi,\nabla_x\omega^0}-\scal{\nabla_xW(x),\nabla_\xi\omega^0}=0
$$
satisfied by the limit $\omega^0$ of $\omega^\eps$. The meaning of this limit needs to be made precise (and shown to exist) and is called the {\em Wigner measure}. It is constant along
Hamiltonian trajectories $\bigl(x(t,y,\eta),\xi(t,y,\eta)\bigr)$ satisfying
$$
\left\{\begin{matrix}
\partial_tx(t,y,\eta)=\xi(t,y,\eta)&x(0,y,\eta)=y\\
\partial_t\xi(t,y,\eta)=-\nabla_xW\bigl(x(t,y,\eta)\bigr)&\xi(0,y,\eta)=\eta
\end{matrix}\right..
$$
For instance, in the case of the harmonic oscillator, $W(x)=\frac{x^2}{2}$, ($d=1$)
we obtain that $\omega^0$ is constant along circles and the circular Radon transform should play
the same role as the X-ray transform. It is thus natural to conjecture the following:

\begin{conjecture}
Let $u_0,v_0\in L^2(\R^d)$ and let $u,v$ be the solution of the
{\em free Schr\"odinger equation}
$$
\left\{\begin{matrix}i\partial_tu(t,x)=-\Delta u(t,x)+|x|^2u(t,x)\\ u(0,x)=u_0\end{matrix}\right.
$$
and
$$
\left\{\begin{matrix}i\partial_tv(t,x)=-\Delta v(t,x)+|x|^2v(t,x)\\ v(0,x)=v_0\end{matrix}\right..
$$
If $|u(t,x)|=|v(t,x)|$ for every $t\in\R$ and every $x\in\R^d$
then there exists $c\in\R$ such that $u_0=cv_0$.
\end{conjecture}

A similar proof as for \cite[Theorem 5.5]{Ja} shows that the answer is positive when $u_0,v_0$
belong to the spectral set of the harmonic oscillator, that is, when they are linear combinations of Hermite functions. It is natural to speculate that the same is true when the potential
$W$ is confining, or at least when it satisfies a lower bound of the form $W(x)\geq c|x|^\alpha$,$c,\alpha>0$.

\section{Schr\"odinger equations on finite graphs}

In this section, we consider $\Gamma=(\vv,\ee)$ to be a finite non-oriented graph. Without loss of generality,
we assume that $\vv=\{1,\ldots,n\}$, we also assume that $n\geq 3$. For $x\in \vv$, we write $y\sim x$ to say that $y\in V$
is a neighbor of $x$ {\it i.e.} $(x,y)\in \ee$ (thus also $(y,x)\in\ee$) and $d(x)$ the number of neighbors
of $x$ (the degree). For a function on $\vv$, the {\em Laplace} operator $\Delta$ is defined by
$$
\Delta u(x)=\left(\sum_{y\sim x}u(y)\right)-d(x)u(x).
$$
Fix a function $W\,:\vv\to\R$ (note that $V$ is {\em real valued}), we say that a function $u\,:\R\times\vv\to\C$
satisfies the {\em Schr\"odinger equation} on $\Gamma$ with potential $V$ if
\begin{equation}
\label{eq:schrgra}
\left\{\begin{matrix}
i\partial_t u(t,x)=-\Delta u(t,x)+W(x)u(t,x)&\ t\in\R,\ x\in\vv,\\
u(0,x)=u_0(x)&\  x\in\vv\end{matrix}\right..
\end{equation}

The aim of this section is to start investigation of Question \ref{Q} in this setting.
This question is still largely unexplored and we will here focus on some simple observation 

\smallskip

Our first observation deals with the connectedness assumption.
Assume that $\Gamma$ has several connected components, $\Gamma=\bigcup_{j=1}^m\Gamma_j$,
$\Gamma_j=(\vv_j,\ee_j)$. Take a function $u_0$ on $\vv$ and write $u_0^j$
for its restriction to $\vv_j$ and let $u_j$ be the corresponding solution of
Schr\"odinger equation on $\Gamma_j$. Write $u=(u_j)_{j=1,\ldots,m}$ for the
function on $\vv$ obtained by gluing back the $u_j$'s. Then $u$ is a solution
of the Schr\"odinger equation on $\Gamma$ with initial condition $u_0$.
In particular, we may attribute an arbitrary phase $c_j\in\T$ to each $u_0^j$
leading to $u_c=(c_ju_j)_{j=1,\ldots,m}$, a solution of  the Schr\"odinger equation on $\Gamma$
that has same modulus as $u$ though $u_c$ is not a multiple of $u$.
From a practical point of view, this is harmless, nevertheless, the connectedness restriction
is necessary to obtain uniqueness up to a global phase factor.

\smallskip

Next, note that $u\to -\Delta u+Vu$ is a {\em real symmetric} operator on $\R^\vv=\R^n$
{\it i.e.} its matrix is real symmetric. Therefore, there is
an orthonormal basis of (real) eigenvectors $(\Phi_j)_{j=1,\ldots,n}$ with corresponding eigenvalues $\lambda_j\in\R$. To further fix ideas, we will assume that
$V\geq 0$, so that the operator is even positive. One can then order the eigenvalues
$0=\lambda_1<\lambda_2\leq\cdots\leq\lambda_n$. The fact that $\lambda_1=0$
and is a simple eigenvalue is a basic fact on spectral graph theory and comes from our connectivity assumption. The corresponding eigenvector is $\phi_1=n^{-1/2}(1,\ldots,1)^t$.
Therefore, if we decompose $u_0$ in this basis
$$
u_0(x)=\sum_{j=1}^na_j\phi_j(x)\quad a_j=\scal{u_0,\phi_j},
$$
we obtain the solution $u(t,x)$ at any time via
$$
u(t,x)=\sum_{j=1}^na_je^{i\lambda_j t}\phi_j(x).
$$

\smallskip

Our second observation is that multiple eigenvalues are a potential source of non uniqueness.

Before elaborating on this a bit further, it should also be noted that this situation is rare. Indeed, Tao and Vu \cite{TV} have shown that
if the graph is drawn randomly (e.g. when considering Erd\"os-Renyi random graphs $G(n,p)$) and the potential is chosen randomly as well
(say $n$ values drawn uniformly at random between $0$ and $1$), then the eigenvalues
are simple with high probability (at least when $n$ is large enough). Note that if
$p\geq2\dfrac{\log n}{n}$ (say) then $G(n,p)$ is also connected with high probability.

Let us now see how we are lead to a classical phase retrieval problem.
Assume that one of the eigenvalues is not simple. Call it $\lambda$
and let $m$ be its multiplicity and $E_\lambda$
be the corresponding eigenspace. Then, if $u_0\in E_\lambda$,
$u(t,x)=e^{i\lambda t}u_0(x)$ which has same modulus as $u_0$.
In this case, we are considering a finite dimension
phase retrieval problem. If $(e_j)_{j=1,\ldots,n}$
is the canonical basis of $\R^n$, we are asking whether $|\scal{u_0,e_j}|$ determines
$u_0$ up to a global phase factor when $u_0\in E_\lambda$.

\begin{lemma}
Let $E$ be a finite dimensional subspace of $\C^m$.
The following are equivalent:
\begin{enumerate}
\renewcommand{\theenumi}{\roman{enumi}}
\item there exist $f,g$ in $E$ such that there coordinates satisfy
$|f_j|=|g_j|$ for every $j$, though $f$ is not a multiple of $g$,

\item there exist $f,g$ in $E$ that are orthogonal and such that there coordinates satisfy
$|f_j|=|g_j|$ for every $j$, though $f$ is not a multiple of $g$,
\end{enumerate}
\end{lemma}

This lemma tells us that we will have non-uniqueness for $u_0\in E_m$
when there are already two orthogonal eigenfunction for the same eigenvalue
that have same modulus (entrywise). It is easy to find two such eigenvectors on the complete graph $K_n$
when $n\geq 3$.

This lemma seems to be known for some time, at least a more evolved version of the lemma can be found in a recent preprint
by Freeman {\it et al} \cite{Fral}. For sake of completeness, here is a proof:

\begin{proof}
One way is obvious, for the second one, assume that $f\not=g$ are such that $|f_j|=|g_j|$
and $\scal{f,g}\not=0$.
Consider the quantity $\min_{|c|=1}\|f-cg\|^2$. Up to replacing $g$ by a unimodular constant
times $g$, we may assume that this minimum is attaigned for $c=1$, that is
$\min_{|c|=1}\|f-cg\|^2=\|f-g\|^2$. In other words, when $|c|=1$,
$$
\|f\|^2+\|g\|^2-2\Re\bar c\scal{f,g}
\geq \|f\|^2+\|g\|^2-2\Re\scal{f,g}.
$$
Now $\Re\bar c\scal{f,g}$ is maximised if and only if $c$ is the phase of $\scal{f,g}$ 
so that $\scal{f,g}$ is real, positive.

Next consider $f_\lambda=f-\lambda\frac{f+g}{2}$ and $g_\lambda=g-\lambda\frac{f+g}{2}$
so that $f_\lambda,g_\lambda\in E$.
Note that $\scal{f_0,g_0}> 0$ while $\scal{f_1,g_1}=-\norm{\frac{f-g}{2}}^2<0$.
Therefore, there exists $\lambda$ such that $\scal{f_\lambda,g_\lambda}=0$.

Further if $z,\zeta\in\C$ are such that $|z|=|\zeta|$
then , for $0\leq\lambda\leq 1$,
\begin{eqnarray*}
\abs{z-\lambda\frac{z+\zeta}{2}}^2
&=&\abs{\left(1-\frac{\lambda}{2}\right)z+\frac{\lambda}{2}\zeta}^2\\
&=&\left(1-\frac{\lambda}{2}\right)^2|z|^2+\frac{\lambda^2}{4}|\zeta|^2\\
&&+\left(1-\frac{\lambda}{2}\right)\lambda\Re(z\bar\zeta).
\end{eqnarray*}
Now as $z$ and $\zeta$ have same modulus, we can exchange them and unwind the
computation to obtain that this is 
$\abs{\zeta-\lambda\dfrac{z+\zeta}{2}}^2$. 
This shows that the coordinates of $f_\lambda$ and of $g_\lambda$ have same modulus.
\end{proof}

The final observation so far is that a stronger situation allows for uniqueness.
Indeed, looking at the $m$-th coordinate of $u(t,x)$, it is of the form
$\dst\sum_{j=1}^na_j\phi_{j,m}e^{i\lambda_j t}$.
The square modulus of this quantity is thus
$$
\sum_{j,k=1}^na_j\overline{a_k}\phi_{j,m}\phi_{k,m}e^{i(\lambda_j-\lambda_k) t}
$$
and this quantity is known for all $t$. 

\begin{definition}
We will say that the multiset $\{\lambda_j,\ j=1,\ldots,n\}$
is {\em totally dissociated} if none of the $\lambda_j$'s is repeated
and if $\lambda_j-\lambda_k\not=\lambda_{j'}-\lambda_{k'}$ if $(j,k)\not=(j',k')$.
\end{definition}

Note that this is a generic condition for a set in $\R^n$. We will assume that the
eigenvalues of $-\Delta+V$ are totally dissociated.
Then, from the linear independence of $\{e^{i\lambda t}\}_{\lambda\in\R}$, we obtain that,
$$
a_j\overline{a_k}\phi_{j,m}\phi_{k,m},\quad j,k,m\in\{1,\ldots,n\}
$$
is uniquely determined by $|u(t,x)|$.

We will need a second property, this time of the eigenvectors:

\begin{definition}
Let $\Gamma=(\vv,\ee)$ be a finite non-oriented connected graph and $V\,:\vv\to\R$
be a potential. Let $\phi_1,\ldots,\phi_n$ be a corresponding orthonormal basis of eigenfunctions. Define the graph $\Gamma_V=(\vv_V,\ee_V)$ with $\vv_V=\{1,\ldots,n\}$
and $(j,k)\in\ee_V$ if and only if there exists $m=m(x,y)\in\{1,\ldots,m\}$ such that
$\phi_{j,m}\phi_{k,m}\not=0$.

We say that $V$ satisfies {\em property (S)}
if $\Gamma_V$ is {\em complete}.
\end{definition}

This situation is generic since the eigenvectors depend continuously on $V$,
any small change in $V$ would thus imply that each $\phi_j$ has full support.

We can now prove our main result in this section:

\begin{theorem}
Let $\Gamma=(\vv,\ee)$ be a non-oriented finite connected graph and $V\,:\rr\to\R$
be a potential. Assume that the eigenvalues of $-\Delta+V$ are totally dissocated
and that $V$ satisfies property (S). Then the modulus $|u(x,t)|$ of the solution
of the Schr\"odinger equation \eqref{eq:schrgra} uniquely determines $u_0$
up to a global phase factor.
\end{theorem}

\begin{proof}
Let $\phi_j$ be the set
let
$$
u_0=\sum_{j=1}^na_j\phi_j\quad v_0=\sum_{j=1}^nb_j\phi_j
$$
be such that $|u(t,x)|=|v(t,x)|$ for all $t\in\R$ and $x\in\vv$.
As said above, from the total dissociativity of the spectrum, for all $j,k,m$,
\begin{equation}
\label{eq:shrgrphasefinal}
a_j\overline{a_k}\phi_{j,m}\phi_{k,m}=
b_j\overline{b_k}\phi_{j,m}\phi_{k,m}.
\end{equation}
Taking $j=k$ in \eqref{eq:shrgrphasefinal}, we obtain 
$|a_j|^2\phi_{j,m}^2=|b_j|^2\phi_{j,m}^2$ for all $m$ and, as at least for one $m$
we have that $\phi_{j,m}\not=0$, we obtain $|a_j|^2=|b_j|^2$.

Now, let $\ell$ be the first $j$ for which $a_j\not=0$ so that $a_j=b_j=0$
for $j<\ell$ and $|a_\ell|^2=|b_\ell|^2\not=0$. Replacing $v$ by $cv$
for some $c\in\T$, that is replacing $(b_j)_j$ by $(cb_j)$,
we may assume that $a_{\ell}=b_{\ell}\not=0$. 
Then, taking $k=\ell$ and $j>\ell$ in \eqref{eq:shrgrphasefinal} we obtain
$a_j\phi_{j,m}\phi_{\ell,m}=b_j\phi_{j,m}\phi_{\ell,m}$.
Our hypothesis is that there is an $m$ such that $\phi_{j,m}\phi_{\ell,m}\not=0$
so that this implies that $a_j=b_j$ for all $j$ and finally that $u_0=v_0$.
\end{proof}

Note that if the hypothesis is not met, the result is false: assume that one of the vertices is not a neighbor to all of the others. Say $1$ is a neighbor of $2,\ldots,m$ but not
of $m+1,\ldots,n$. Take $a_1=b_1=1$, $a_j=b_j=0$ for $j=2,\ldots,m$
 and $a_{j}=-b_j=1$ for $j=m+1,\ldots,n$. Then
$$
a_j\overline{a_k}\phi_{j,m}\phi_{k,m}=\begin{cases}\phi_{j,m}\phi_{k,m}&\mbox{when }j=k=1\mbox{ or }j,k\geq m+1\\
0&\mbox{otherwise}\end{cases}
$$
and is therefore $=b_j\overline{b_k}\phi_{j,m}\phi_{k,m}$. Thus, if 
$u_0=\phi_1+\phi_{m+1}+\cdots+\phi_n$ and $v_0=\phi_1-\phi_{m+1}-\cdots+\phi_n$
then the corresponding solutions have same modulus at all time.

In this case, one may nevertheless obtain results, but no longer for every $u_0$.
An inspection of the above proof and of the counterexample show that
the problems come from $a_j$'s that are $0$. It is actually enough
to have an $a_j\not=0$ so that in the graph $\Gamma_V$, $j$ is a neighbor of every other vertex.
When $\Gamma$ is connected and $V=0$, $\phi_1=n^{-1/2}(1,\ldots,1)^t$, is an eigenvector
(corresponding to the $0$ eigenvalue) so that all edges are connected to $1$ in $\Gamma_0$.
In particular, if $u_0=(u_{0,j})_{j=1,\ldots,n}$ is such that
$$
a_1=\scal{u_0,\phi_1}=n^{-1/2}\sum_{j=1}^nu_{0,j}\not=0
$$
then $|u(x,t)|$ uniquely determines $u_0$.

One may perturb this result to show that if $V$ is small enough, there exists $\eps_V>0$
such that, if $\dst\abs{\sum_{j=1}^nu_{0,j}}>\eps_V$ then 
$|u(x,t)|$ still uniquely determines $u_0$.

\end{document}